\documentclass[a4paper, 11 pt]{article}

\usepackage{amsfonts,amsmath,amssymb,hhline,ifthen,amsthm,graphics,cite,latexsym, caption2}
\usepackage{latexsym}
\usepackage[T2A]{fontenc}
\usepackage[cp1251]{inputenc}
\usepackage[english]{babel}

\usepackage[matrix, arrow, curve]{xy}
\usepackage{graphicx}

\newcounter{q3}
\newcommand{\dsm}[3]{
{\if#20{\if#31{\frac{\partial #1}{\partial y}}\else
          {\frac{\partial^{#3} #1}{\partial y^{#3}}}
        \fi}\else
  {\if#30{\if#21{\frac{\partial #1}{\partial x}}\else
            {\frac{\partial^{#2} #1}{\partial x^{#2}}}
          \fi}\else
    {\setcounter{q3}{#2}\addtocounter{q3}{#3}
    \frac{\partial{\if{1}\arabic{q3}^{}\else^{ \arabic{q3} }\fi}#1}
    {{\if#20\else{\partial x{\if#21\else{^{#2}}\fi}}\fi}
     {\if#30\else{\partial y\if#31\else{^{#3}}\fi}\fi} }}
   \fi}
\fi} }

\newcommand{\dzm}[2]{
{\if#10{\if#21{\frac{\partial}{\partial\overline{\zeta}}}\else
          {\frac{\partial^{#2}}{\partial\overline{\zeta}^{#2}}}
        \fi}\else
  {\if#20{\if#11{\frac{\partial}{\partial\zeta}}\else
            {\frac{\partial^{#1}}{\partial\zeta^{#1}}}
          \fi}\else
    {\setcounter{q3}{#1}\addtocounter{q3}{#2}
    \frac{\partial{\if{1}\arabic{q3}^{}\else^{ \arabic{q3} }\fi}}
    {{\if#10\else{\partial\zeta{\if#21\else{^{#1}}\fi}}\fi}
     {\if#20\else{\partial\overline{\zeta}\if#21\else{^{#2}}\fi}\fi} }}
   \fi}
\fi} }

\newcounter{q2}

\newtheoremstyle{theor}
  {\medskipamount}
  {\medskipamount}
  {\itshape}
  {}
  {\bfseries}
  {.}
  {.5em}
  {}

\newtheorem{definition}{Definition}[section]
\newtheorem{theorem}[definition]{Теорема}
\newtheorem{lemma}[definition]{Лемма}
\newtheorem{proposition}[definition]{Предложение}
\newtheorem{corollary}[definition]{Следствие}

\theoremstyle{definition}
\newtheorem{remark}[definition]{Замечание}

\numberwithin{equation}{section}

\makeindex

\newtheoremstyle{remarks}
  {0mm}
  {0mm}
  {\itshape}
  {}
  {\itshape}
  {.}
  {.5em}
  {}
\makeatletter \@addtoreset{equation}{section} \makeatother

\begin{document}

\subsection*{\center ОБОБЩЕННЫЕ КВАЗИИЗОМЕТРИИ НА ГЛАДКИХ РИМАНОВЫХ МНОГООБРАЗИЯХ}\begin{center}\textbf{Е.С.~Афанасьева} \end{center}
\parshape=5
1cm 13.5cm 1cm 13.5cm 1cm 13.5cm 1cm 13.5cm 1cm 13.5cm \noindent \small {\bf Аннотация.} В работе изучается граничное поведение конечно билипши\-цевых ото\-бра\-жений
 на гладких римановых многообразиях.

 \parshape=2
1cm 13.5cm 1cm 13.5cm  \noindent \small {\bf Ключевые слова:} гладкие римановы многообразия, конечно билипшицевы отобра\-жения, модули, нижние $Q$-гомеоморфизмы.

\bigskip
\textbf{ AMS 2010 Subject Classification: 30L10, 30C65}


\large\section {Введение} Напомним некоторые определения из теории
римановых многообразий, которые можно найти, например, в монографиях
\cite{K},   \cite{PSh}, \cite{L}  и \cite{R}. На\-помним, что {\it
n-мерное топологическое многообразие $\mathbb{M}^{n}$} -- это
хаусдорфово тополо\-гическое пространство со счетной базой, в котором
каждая точка имеет открытую окрестность, гомеоморфную
$\mathbb{R}^{n}$. {\it Картой на многообразии $\mathbb{M}^{n}$}
называется пара $(U,\varphi)$, где $U$ -- открытое подмножество
пространства $\mathbb{M}^{n}$, а $\varphi$ -- гомеоморфное
отображение подмножества $U$ на открытое под\-множество координатного
пространства $\mathbb{R}^n$, с помощью которого каждой точке $p \in
U$ ставится во взаимно однозначное соответствие набор из $n$ чисел,
ее {\it локальных координат}. Полный набор всех карт многообразия
называется его {\it атласом}. {\it  Гладкое многообразие} --
многообразие с картами $(U_{\alpha},\varphi_{\alpha})$, локальные
координаты которых связаны гладким ($C^{\infty}$) образом.

{\it Римановым многообразием} $(\mathbb{M}^n,g)$ называется гладкое
многообразие с заданной на нем римановой метрикой, т.е.
  положительно определенным симметричным тензорным полем $g\ =\ g_{ij}(x),$
  которое определяется  в координатных картах с правилом перехода:
$$
'g_{ij}(x)=g_{kl}(y(x))\ \frac{\partial y^{k}}{\partial
x^{i}}\frac{\partial y^{l}}{\partial x^{j}}\ \ ,
$$ где, как обычно, $k$ и $l=1,...,n$ -- так называемые связанные индексы, по
которым производится суммирование.
$g_{ij}(x)$ в дальнейшем подразумевает\-ся гладким. Заметим, что $ det
g_{ij}>0$ в силу положительной определенности $g_{ij}$, см.,  напр.,
\cite[c. 277]{Ga}. 

 {\it Элемент длины} задается инвариантной дифференциальной формой
$$ds^{2}={g_{ij}dx^{i}dx^{j}}:=\sum_{i,j=1}^{n}g_{ij}dx^{i}dx^{j},$$
  где $g_{ij}$ -- метрический тензор,  $x^{i}$ -- локальные координаты.
В соответствии с этим, если $\gamma:[a,b]\rightarrow \mathbb{M}^n$
-- кусочно-гладкая кривая и $x(t)$ -- ее параметри\-ческое задание в
локальных координатах, то ее длина вычисляется по формуле:
$$
s_\gamma = \int\limits_{a}^{b}\sqrt{g_{ij}(x(t))\
\frac{dx^i}{dt}\frac{dx^j}{dt}}\ dt
$$ {\it Геодезическое расстояние $d(p_1,p_2)$} определяется как
инфимум длин кусо\-чно-глад\-ких кривых, соединяющих точки $p_1$ и $p_2$
в $(\mathbb{M}^n,g)$. Любая кривая, соединяющая  $p_1$ и $p_2$, на
которой реализуется этот инфимум, называется {\it геодезической}.

 Напомним также, что {\it элемент объема} на $(\mathbb{M}^n,g)$  определяется инва\-риант\-ной
  формой $ dV =  \sqrt{det{g_{ij}}} \ dx^{1}...dx^{n},$
а {\it элемент площади} гладкой поверхности $H$ на
$(\mathbb{M}^n,g)$   $d{\cal{A}}=\sqrt{det \
{g_{\alpha\beta}^{*}}} \ du_{{1}}...du_{{n-1}}$ -- инвариантной фор\-мой, где
$g_{\alpha\beta}^{*}$ -- риманова метрика на $H$, порожденная
исходной  римановой метрикой $g_{ij}$ по формуле:

$$ g_{\alpha\beta}^{*}(u)=g_{ij}(x(u))\cdot\frac{\partial x^{i}}
 {\partial u^{\alpha}}\cdot\frac{\partial x^{j}}{\partial u^{\beta}}.
$$ Здесь $x(u)$ -- гладкая
 параметризация поверхности $H$ с $\nabla_{u}x\neq 0$ всюду.
Таким образом, метрический тензор $g$ на римановом многообразии
по\-рождает соответствующий метрический тензор $g^{*}$ на
произвольной регу\-ляр\-ной поверхности, см., напр., $\S$ 88 в \cite{R}.

Здесь под {\it поверхностью} на многообразии $(\mathbb{M}^n,g)$
понимается непре\-рывное отображение $H:U\rightarrow \mathbb{M}^n,$
где $U$ -- область в $(n-1)$-мерном пространстве $\mathbb{R}^{n-1}$
или, более общо, $U$ -- $(n-1)$-мерное многообразие, например,
$(n-1)$-мерная сфера. Если отображение $H$ является гладким в
локальных координатах, то поверхность называют {\it гладкой}.
Например, геодезическая сфера в достаточно малой окрестности
произвольной точки гладкого риманова
многообразия -- гладкая поверхность, см. монографию \cite[с. 106]{L}.

Для нас важны следующие фундаментальные факты, см., напр., лемму
5.10 и следствие 6.11 в \cite{L}, а также \cite[с. 260 - 261]{K}.
\begin{proposition}\label{prop1} {\it Для каждой  точки  риманова  многообразия
существу\-ют  ее окрестности и соответствующие локальные координаты в
них, в которых геодезическим сферам с центром в данной точке
соответствуют евклидовы сферы с теми же радиусами и с центром в
начале координат, а связке геодезических, исходящих из данной точки,
соответствует связка лучей, исходящих из начала координат.}\end{proposition}

 Указанные окрестности и координаты принято называть {\it нормальными}.

\begin{remark} \label{rem1} В частности, в нормальных координатах геодезические сферы имеют
естественную гладкую параметризацию через направляющие косинусы
соответствующих лучей, исходящих из начала координат. Кроме того,
метрический тензор в начале координат в этих координатах совпадает с
единичной матрицей, см., напр., предложение 5.11 в \cite{L}.\end{remark}


\section {О нижних $Q$-гомеоморфизмах на гладких римановых
многообразиях} В дальнейшем мы используем обозначения геодезических
сфер $S(P_0, \varepsilon)=\{P\in \mathbb{M}^n: d(P,P_0)=
\varepsilon\}$, геодезических шаров $B(P_0,\varepsilon)=\{P\in
\mathbb{M}^n: d(P,P_0)< \varepsilon\}$ и геодезических колец
$\mathbb{A}(P_0, \varepsilon,\varepsilon_0)=\{P \in \mathbb{M}^n:
\varepsilon<d(P,P_0)<\varepsilon_0\}$, где $d$ -- геодезическое
расстояние на $(\mathbb{M}^n,g)$ и подразумеваем, что
 $S(P_0, r)$, $B(P_0, r)$ и $\mathbb{A}=\mathbb{A}(P_0, \varepsilon,\varepsilon_0)$
 лежат в нормальной окрестности точки
$P_0$. Далее для любых множеств $A$, $B$  и $C$ в
$(\mathbb{M}^n,g)$, $n\geq2$, через $\triangle (A,B;C)$ обозначаем
семейство всех кривых $\gamma:[a,b]\rightarrow \mathbb{M}^n,$
соединяющих $A$ и $B$ в $C$, т.е. $\gamma(a)\,\in\,A\,,\gamma(b)
\,\in\,B$ и $\gamma(t)\,\in\,C\,$ при $a\,<\,t\,<\,b\,.$

Борелевскую функцию $\rho:\mathbb{M}^n\to [0,\infty]$ называем {\it
допустимой} для семейства $\Gamma$ поверхностей $S$ в
$\mathbb{M}^n$,  пишем $\rho\in adm\,\Gamma$, если
\begin{equation}\label{44.1.33}
 \int\limits_S\rho^{n-1}\,d{\cal A}\geq1\ \ \ \ \forall\  S \in
 \Gamma.
 \end{equation}
{\it $p$-модуль семейства  поверхностей  $\Gamma$} при $p \in (0,\infty)$ есть
величина
$$
M_p(\Gamma):= \inf\limits_{\rho\ \in\  adm \ \Gamma} \int
\limits_{\mathbb{M}^n}\rho^p\  dV.$$ При $p=n$ получаем конформный модуль и в этом случае используем обозначение $M(\Gamma)$.

Аналогично статье \cite{KR1}, см. также монографию \cite{MRSY}, борелеву функцию
$\rho:\mathbb{M}^n\to[0,\infty]$ называем {\it обобщенно допустимой}
 для семейства $\Gamma$ поверх\-ностей $S$ в
$\mathbb{M}^n$ относительно $p$-модуля, пишем $\rho\in ext_p\,adm\,\Gamma$, если условие
допустимости (\ref{44.1.33}) выполнено для $p$-почти всех ($p$-п.в.)
$S\in\Gamma$, т.е. за исключением подсемейства $\Gamma$ нулевого $p$-модуля. При $p=n$ пишем  $\rho\in ext\,adm\,\Gamma.$

Следующее понятие, мотивированное кольцевым определением Геринга для
квазиконформных отображений в \cite{G}, было впервые введено в
$\mathbb{R}^{n}$, $n\geq2,$ в \cite{KR1}.

 Пусть всюду далее $D$ и $D_*$ -- области на гладких римановых многообрази\-ях
 $(\mathbb{M}^n,g)$ и  $(\mathbb{M}^{n}_*,g^{*})$, $n\geq2,$ соответственно,
  $Q:\mathbb{M}^n \to(0,\infty)$ -- измеримая функция. Гомеоморфизм
$f:D\to D_*$ будем называть {\it нижним ${\it Q}$-гомео\-морфизмом
относительно p-модуля в точке $P_0\in \overline{D}$}, если существует $\delta_0
\in (0, d(P_0))$, $d_0:=\sup\limits_{P\in D}d(P_0,P),$ такое что для всякого
$\varepsilon_0<\delta_0$
 и любых геодезических колец
$\mathbb{A}=\mathbb{A}(P_0,\varepsilon,\varepsilon_0),\
\varepsilon\in(0,\varepsilon_0),$ выполнено условие
\begin{equation}\label{44.1.122} M_p(f(\Sigma_{\varepsilon}))\
\geq\ \inf\limits_{\rho\ \in\  ext_p\
adm\,\Sigma_{\varepsilon}}\int\limits_{\mathbb{A} \  \cap D}
\frac{\rho^p(P)}{Q(P)}\ dV,\end{equation}  где через
$\Sigma_{\varepsilon}$ обозначено семейство пересечений всех геодезических сфер
$S(P_0,r), r\in(\varepsilon,\varepsilon_0),$ с областью $D$. Будем также говорить,
что гомеоморфизм $f:D\to D_*$ является {\it нижним
$Q$-гомео\-морфиз\-мом относительно p-модуля}, если $f$ является
нижним $Q$-гомео\-морфиз\-мом относительно p-модуля в каждой точке
$P_0 \in \overline{D}$, ср. с \cite{MRSY}.

Следующий критерий нижних $Q$-гомеоморфизмов, см. теорему 2.1 в \cite{KR1} и теорему  9.2 в \cite{MRSY}, впервые был доказан для конформного модуля
в $\mathbb{R}^{n}$ при $n\geq2$;  при $p\neq n$ в $\mathbb{R}^{n}$, см. теорему 9.2 в \cite{GS}, а также для гладких римановых многообразий
 $(\mathbb{M}^n,g)$ при $n\geq2$ относительно конформного модуля,
 см. теорему 4.1 в \cite{ARS}.

\begin{theorem} \label{th1} {\it Пусть $D$  и  $D_*$ -- области на гладких
римановых много\-образиях $(\mathbb{M}^n,g)$ и
$(\mathbb{M}^{n}_*,g^{*})$, $n\geq2$, соответственно,
  $Q:\mathbb{M}^n\to(0,\infty)$ -- измеримая функция, и $P_0\,\in\,\overline{D}$.
  Гомеоморфизм $f:D\to
D_*$ является нижним $Q$-гомеоморфизмом относительно p-модуля, $p>n-1$, в точке $P_0$,
тогда и только тогда, когда для любой нормальной
окрестности $B(P_0,\varepsilon_0)$ точки $P_0$ с
$0<\varepsilon_0<d_0:=\sup\limits_{P\in D}d(P_0,P)$
\begin{equation} \label{11.1} M_p(f(\Sigma_{\varepsilon}))\ \geq\
\int\limits_{\varepsilon}^{\varepsilon_0}
\frac{dr}{||\,Q||\,_{s}(P_0,r)}\,\quad\forall\
\varepsilon\in(0,\varepsilon_0),\end{equation} где $s=\frac{n-1}{p-n+1}$,
$\Sigma_{\varepsilon}$
 -- семейство всех пересечений с областью $D$
геодезических сфер $S(P_0,r)$, $r\in(\varepsilon,\varepsilon_0),$ и
\begin{equation} \label{11.10} ||\,Q||\,_{s}(P_0,r)=\left(\int\limits_{D(P_0,r)}Q^{s}(P)\
d{\cal A}\right)^{\frac{1}{s}}\end{equation} -- $L^{s}$-норма
$Q$ по $D(P_0,r)=\{P\in D: d(P,P_0)=r\}=D\cap S(P_0,r)$.}\end{theorem}

Заметим, что инфимум в (\ref{44.1.122}) достигается для функции
$$\rho_0(P)\ =\left[\frac{Q(P)}{||\,Q||_{s}(P_0, d(P,P_0))}\right]^{\frac{1}{p-n+1}}.$$

Таким образом, неравенство (\ref{11.1}) является точным для нижних $Q$-гомео\-мор\-физ\-мов относительно $p$-модуля.

\begin{proof} Отметим, что  в (\ref{44.1.122}) по теореме Лузина, предложению \ref{prop1} и замечанию \ref{rem1}
$$\inf\limits_{\rho \ \in   \ adm\
\Sigma_{\varepsilon}}\int\limits_{\mathbb{A}\ \cap D}\frac{\rho^p(P)}{Q(P)}\
dV\ =\inf\limits_{\rho \ \in  \ ext_p \ adm\
\Sigma_{\varepsilon}}\int\limits_{\mathbb{A}\ \cap D}\frac{\rho^p(P)}{Q(P)}\
dV.$$
Также для любого $\rho \in ext_p \ adm\
\Sigma_{\varepsilon},$
$$\widetilde{A}_\rho (r)\ :=
\int\limits_{D(P_0,r)} \rho^{n-1}(P)\ d{\cal{A}}\ \neq \ 0 \ \ \ \text{п.в.}$$
 является измеримой функцией по параметру $r$, скажем по теореме
Фубини, предложению \ref{prop1} и замечанию \ref{rem1}.  Таким
образом, мы можем требовать равенство $\widetilde{A}_\rho (r)= 1$
п.в. вместо условия допустимости (\ref{44.1.33}) и
$$\inf\limits_{\rho \ \in  \ ext_p \ adm\
\Sigma_{\varepsilon}}\int\limits_{\mathbb{A}\ \cap D}\frac{\rho^p(P)}{Q(P)}\
dV=
\int\limits_{\varepsilon}^{\varepsilon_0}\left(\inf\limits_{\widetilde{\alpha}\in\
I(r)} \int\limits_{D(P_0,r)}\frac{\widetilde{\alpha}^q(P)}{Q(P)}\
d{\cal{A}}\right)dr,$$ где $q=p/ (n-1)>1,$
$\mathbb{A}=\mathbb{A}(P_0,\varepsilon,\varepsilon_0)$, $\varepsilon\in
(0,\varepsilon_0),$ и $I(r)$ обозначает множество всех борелевых
функций $\widetilde{\alpha}$ на поверхности $D(x_0,r),$ таких, что
$$\int\limits_{D(P_0,r)}\widetilde{\alpha}(P)\ d{\cal{A}}=1.$$
Поэтому теорема 1 следует из леммы 2.1 в \cite{KR1}, см. также лемму 9.2 в \cite{MRSY} для
$X=D(P_0,r)$ с мерой площади на $D(P_0,r)$ в качестве $\mu,\
\varphi=\frac{1}{Q}|_{D(P_0,r)}$ и $q=p/(n-1)>1$.
\end{proof}

\medskip

Следующий результат сначала был  доказан  на гладких римановых много\-образиях $(\mathbb{M}^n,g)$, $n\geq2$,
относительно конформного модуля, см.  лемму 4.1 в \cite{ARS}.

\begin{lemma} \label{lem1} \label{lem1} \ {\it  Пусть  $D$   и  $D_*$  --  области  на
гладких \ римановых\  многообразиях \ $(\mathbb{M}^n,g)$  и
$(\mathbb{M}^n_*,g_*)$, $n\geq2$, $Q:\mathbb{M}^n \rightarrow (0,
\infty)$ -- измеримая функция и $f:D\rightarrow D_*$ -- нижний
$Q$-гомеоморфизм относительно p-модуля в точке $P_0\in \overline{D}$, $p>n-1$.
Тогда
\begin{equation}\label{1}
M_\alpha\left(\Delta(f(S_\varepsilon),f(S_{\varepsilon_0}); D_*)\right)\leq c/I^{s},
\end{equation}
 где $\alpha=\frac{p}{p-n+1}$, $s=\frac{n-1}{p-n+1}$,
 $S_\varepsilon=S(P_0, \varepsilon)$ и $S_{\varepsilon_0}=S(P_0, \varepsilon_0)$,
 $0<\varepsilon<\varepsilon_0$, $B(P_0,\varepsilon_0)$ -- нормальная окрестность точки
 $x_0$,
\begin{equation}\label{1.7}
I=I(P_0,\varepsilon,\varepsilon_0)=
\int\limits_{\varepsilon}^{\varepsilon_0}\frac{dr}{||Q||_{s}(P_0,r)},
\end{equation}
$||Q||_{s}(P_0,r)$ определено в (\ref{11.10}), а константа $c$
произвольно близка к 1 в достаточно малых
окрестностях точки $P_0$.}\end{lemma}

\begin{proof}  Учитывая тот факт, что по замечанию \ref{rem1}
метрический тензор в начале нормальных координат  совпадает с
единичной матрицей и, следо\-вательно, в достаточно малом шаре с
центром в нуле равномерно близок к единичной матрице, получаем,
согласно равенствам Хессе и Циммера, см. \cite{Hes}  и \cite{Zi},
что
\begin{equation}\label{1.6}
M_\alpha\left(\Delta(f(S_{\varepsilon}),f(S_{\varepsilon_0});
D_*))\right)\leq \frac{c}{M_p^{s}(f(\Sigma))},
\end{equation} $\alpha=\frac{p}{p-n+1}, $ $1<\alpha<\infty$, $n-1<p<\infty$, поскольку
$f(\Sigma)\subset \Sigma (f(S_{\varepsilon}),f(S_{\varepsilon_0});
D_*),$ где $\Sigma$ обозначает совокупность всех геодезических сфер
с центром в точке $P_0$, расположенных между сферами $S_{\varepsilon}$ и $S_{\varepsilon_0}$, а
$\Sigma (f(S_{\varepsilon}),f(S_{\varepsilon_0}); D_*)$ состоит из
всех замкнутых множеств в $D_*$, отделяющих $f(S_{\varepsilon})$ и
$f(S_{\varepsilon_0})$, а $c$ -- постоянная, произвольно близкая к
единице в достаточно малых окрестно\-стях $P_0$. Таким образом, из
теоремы \ref{th1} и соотношения (\ref{1.6}) получаем оценку
(\ref{1}), где интеграл $I=I(P_0,\varepsilon,\varepsilon_0)$
определен в (\ref{1.7}).
\end{proof}

\medskip

Аналог приведенной ниже леммы   был  ранее получен в $(\mathbb{M}^n,g)$,
$n\geq2$, относительно конформного модуля, см. лемму 3 в \cite{AS}.

\begin{lemma} \label{lem2} {\it  Пусть $D$  -- область на гладком
римановом многообразии $(\mathbb{M}^n,g)$, $n\geq2$, $P_0 \in \overline{D}$,
$0<\varepsilon<\varepsilon_0<d_0:=\sup\limits_{P\in D}d(P,P_0)$, $\mathbb{A}=
\mathbb{A}(P_0,\varepsilon,\varepsilon_0)$ -- геодезическое кольцо,
$B(P_0,\varepsilon_0)$ -- нормальная окрестность точки $P_0$  и
пусть $Q:\mathbb{M}^n \rightarrow (0, \infty)$ -- измеримая функция,
которая интегрируема в степени $s$, $s=\frac{n-1}{p-n+1}$, где $p>n-1$ в $B(P_0,\varepsilon_0)$. Пусть
$$
\eta_0(t)=\frac{1}{I\cdot\|Q\|_{s}(P_0,t)},
$$где
$\|Q\|_{s}(P_0,r)$, $r \in (\varepsilon,\varepsilon_0),$ и
$I=I(P_0,\varepsilon,\varepsilon_0)$ определены в (\ref{11.10}) и
(\ref{1.7}), соответственно. Тогда
\begin{equation}\label{2.2}
1/I^{s}=\int\limits_{\mathbb{A}\ \cap D}
Q^{s}(P)\cdot\eta_0^\alpha\left(d(P,P_0)\right)\
dV\leq\int\limits_{\mathbb{A}\ \cap D}
Q^{s}(P)\cdot\eta^\alpha\left(d(P,P_0)\right)\ dV,
\end{equation} где $\alpha=\frac{p}{p-n+1}$ для любой борелевой функции
$\eta:(\varepsilon,\varepsilon_0)\to [0,\infty]$, такой, что
\begin{equation}\label{2.3}
 \int\limits_{\varepsilon}^{\varepsilon_0}\eta(r)dr=1.\end{equation}} \end{lemma}

\begin{proof} В дальнейшем мы пользуемся тем
обстоятель\-ством, что по замечанию \ref{rem1}
 элементы объема и площадей на геодезических сферах в
нормальных окрестностях точки $P_0$ экви\-валентны евклидовым с
коэффициентом экви\-валентности произвольно близким к единице  в
достаточно малых окрест\-ностях, а радиусы геодезических сфер
$S(P_0,r)$ совпадают с евклидовыми.

Если $I=\infty$, то левая часть соотношения (\ref{2.2}) равна нулю
 и неравенство в этом случае очевидно. Заметим, что если $I=0$, то
 $\|Q\|_{s}(P_0,r)=\infty$ для п.в. $r\in (\varepsilon,\varepsilon_0),$ что невозможно ввиду интегрируемости $Q^s$ в $B(P_0,\varepsilon_0)$. Поэтому можно считать, что $0<I<\infty$. Тогда
 $\eta_0(r)\neq\infty$ п.в. в
$(\varepsilon,\varepsilon_0)$, поскольку $\|Q\|_{s}(P_0,r)\neq0$ п.в. Кроме  того, $\|Q\|_{s}(P_0,r)\neq\infty$ п.в. поскольку $Q\in L^s(B(P_0,\varepsilon_0)).$ Полагая

 $$\beta(r)=\eta(r)\cdot\|Q\|_{s}(P_0,r)$$ и $$
\omega(r)=[\|Q\|_{s}(P_0,r)]^{-1},
$$ будем иметь, что
$\eta(r)=\beta(r)\omega(r)$ п.в. в
$(\varepsilon,\varepsilon_0)$ и что
$$C:=\int\limits_{\mathbb{A}\ \cap D}
Q^{s}(P)\cdot\eta^\alpha\left(d(P,P_0)\right)\
dV=\int\limits_{\varepsilon}^{\varepsilon_0}\beta^{\alpha}(r)\omega(r)dr.$$

Применяя неравенство Иенсена с весом, см., напр., теорему 2.6.2 в
\cite{Ran}, к выпуклой функции $\varphi(t)=t^{\alpha}$, заданной в
интервале $\Omega=(\varepsilon,\varepsilon_0)$, с вероятностной
мерой
$$\nu(E)=\frac{1}{I}\int\limits_{E}\omega(r)\ dr,$$
получаем что $$\left(\frac{1}{I}\int\limits_{\varepsilon}^{\varepsilon_0}
\beta^{\alpha}(r)\omega(r)dr\right)^{\frac{1}{\alpha}}\geq
\frac{1}{I}\int\limits_{\varepsilon}^{\varepsilon_0} \beta(r)\omega(r)\ dr =\frac{1}{I},$$ где мы
также использовали тот факт, что $\eta(r)={\beta(r)}\omega(r)$
удовлетворяет соотношению (\ref{2.3}). Таким образом, $$C\geq
\frac{1}{I^{s}},$$ что и доказывает (\ref{2.2}).\end{proof}

\begin{corollary} \label{cor1} {\it  При условиях и обозначениях лемм \ref{lem1} и
\ref{lem2},
\begin{equation} \label{1.10}
M_\alpha(\Delta(f(S_\varepsilon),f(S_{\varepsilon_0}); D_*))\leq
c\int\limits_{\mathbb{A}\ \cap D}Q^{s}(P) \ \eta^{\alpha}(d(P,P_0))\ dV,
\end{equation} где $S_\varepsilon=S(P_0,\varepsilon)$ и $S_{\varepsilon_0}=S(P_0,\varepsilon_0).$}
\end{corollary}

Другими словами, это означает, что нижний $Q$-гомеоморфизм  отно\-сительно $p$-мо\-ду\-ля в $\mathbb{R}^n, \ n\geq2,$ с $Q\in L^s(B(P_0,\varepsilon_0))$, $s=\frac{n-1}{p-n+1}$,  является $Q_*$-кольцевым гомеоморфизмом относительно $\alpha$-модуля с $Q_*=Q^s$, $\alpha=\frac{p}{p-n+1}.$ Ясно, что $\alpha>p$ при $n\geq2$, \cite{GS}.

Отметим, что теория нижних $Q$-отображений применима к отображениям с конечным искажением класса Орлича-Cоболева $W^{1,\varphi}_{loc}$ при наличии условия Кальдерона и, в частности, к классам Соболева $W^{1,p}_{loc}$ при $p > n-1$ (см. \cite{ARS},\cite{KRSS3}--\cite{KRSS1}). В работах \cite{KPRS}, \cite{Sal}  также  приводятся приложения нижних $Q$-гомеоморфизмов к исследованию локального и граничного поведения гомеоморфных решений с обобщенными производными и к задаче Дирихле для уравнений Бельтрами с вырождением.


\section {Об обобщенных квазиизометриях} Говорим, что  отображение $f:\mathbb{M}^n\to
\mathbb{M}^{n}_*$, $n\geq2$, называется {\it липшицевым}, если для
некоторого $L<\infty$ и для всех $P,\ T$ из $(\mathbb{M}^n,g)$,
выполнено неравен\-ство
$$d_*(f(P),f(T))\leq L\,d(P,T),$$
где $d$ и $d_*$ -- геодезические расстояния на $(\mathbb{M}^n,g)$ и
$(\mathbb{M}^{n}_*,g^{*})$, соответственно. Наименьшая \ из \ таких\
констант\  называется \ {\it константой\  Липшица}\  и обозначается
$Lip(f)$. Одним из примеров липшицевой функции в $\mathbb{R}^n$ может служить
функция $f(x)=dist(x,F)$, где $F$ -- замкнутое подмножество $\mathbb{R}^n$,
причем, $Lip(f)=1$.

Существует также и более узкий класс отображений, чем липшицевы, а именно, билипшицевы отображения.

Говорим также, что отображение $f:\mathbb{M}^n\to
\mathbb{M}^{n}_*$, $n\geq2$,  {\it билипшицево}, если оно, во-первых, липшицево, а во-вторых,
$$L^{*}\,d(P,T)\leq d_*(f(P),f(T))$$ для
некоторого $L^{*}>0$ и для всех $P$ и $T$ из $(\mathbb{M}^n,g)$.

Пусть далее  $\Omega$ -- открытые множества на $(\mathbb{M}^n,g)$,
$n\geq2$,  $f:\Omega\to \mathbb{M}^{n}_*$ -- непрерывное
отображение.
Аналогично \cite{KPR}, см. также  \cite{MRSY} говорим, что отображение
$f:\Omega\to \mathbb{M}^{n}_*$  {\it конечно липшицево},
если
$L(P,f)<\infty$ для всех $P\in\Omega$, и {\it конечно
билипшицево}, если
$$0<l(P,f)\leq L(P,f)<\infty$$
для всех $P\in\Omega$, где
\begin{equation}  \label{13} L(P,f)=\limsup_{T\to
P}\frac{d_*(f(P),f(T))}{d(P,T)},\end{equation} $T\in \mathbb{M}^n,$ и
$$l(P,f)=\liminf_{T\to
P}\frac{d_*(f(P),f(T))}{d(P,T)}.$$

Очевидно, что каждое липшицево отображение является конеч\-но
липши\-цевым и, соответственно, каждое билипшицево отображение
является конеч\-но билипшицевым.

Обозначим далее через $K_p(P,f)$ {\it внешнюю дилатацию}  на
$(\mathbb{M}^n,g)$, $n\geq2$, определяемую следующим образом:
\begin{equation}  \label{12} K_p(P,f)= \left \{\begin{array}{ll}
\frac{L^p(P,f)}{J(P,f)} \qquad &{\text {при}}  \ J(P,f)\ne 0,
\\ 1  & {\text {при}}  \ L(P,f)=0,
\\ \infty  & {\text {в остальных точках}}\end{array}
\right., \end{equation} где \begin{equation} \label{11}
J(P,f)=\lim\limits_{r\rightarrow 0} \frac{V_{*}(f(B(P,r)))}{V(B(P,r))} \ \ \ \ \text{п.в.}
 \end{equation}

Напомним, что при $p=n$ получаем внешнюю дилатацию $K(P,f)$ отобра\-же\-ния $f$, определенную стандартным образом, см., напр., п. 3 в \cite{ARS}. Напомним также, что    $\Vert
f^{\,\prime}(x)\Vert$ в $\mathbb{R}^{n}$ обозначает матричную норму якобиевой матрицы
$f^{\,\prime}$ отображения $f$ в точке $x\in D$,
$\Vert f^{\,\prime}(x)\Vert=\sup\limits_{h\in{\Bbb
R}^n,|h|=1}|f^{\,\prime}(x)\cdot h|\,,$ $J_f(x)={\rm
det}\ f^{\,\prime}(x)$ -- якобиан отображения $f$ и $K_{f}(x)=\Vert
f^{\,\prime}(x)\Vert^n/J_f(x)$ -- внешнюю дилатацию отображения $f$ в $\mathbb{R}^{n}$.

\begin{remark} \label{rem2} Переходя к локальным координатам, по замечанию
1 видим, что определения    $L(P,f)$ из (\ref{13}) и $\|f'(x)\|$  в $\mathbb{R}^{n}$, внешней дилатации
 $K_p(P,f)$ из (\ref{12}) и $K_f(x)$  в $\mathbb{R}^{n}$, а также обобщенного якобиана $J(P,f)$ из (\ref{11}) и $J_f(x)$ из $\mathbb{R}^{n}$, соответственно,  согласованы  в
точках дифференци\-руемости отображения $f$. Заметим также, что
величина $K_f(x)$ инвариантна относительно замен локальных координат. Таким образом, как видно из нормальных координат, $K_p(P,f)$ можно вычислять п.в. через
$K_f(x)$ и
в любых локальных координатах для указанных отображений.\end{remark}

Следующее утверждение является ключевым для дальнейшего иссле\-дования.

Впервые аналогичный результат был получен в $\mathbb{R}^{n}$, $n\geq 2$, см. следствие 5.15 в \cite{KPR},
см. также следствие 10.10 в \cite{MRSY}.

\begin{theorem} \label{2} {\it   Пусть $(\mathbb{M}^n,g)$ и
$(\mathbb{M}^{n}_*,g^{*})$, $n\geq 2$, -- гладкие римановы
много\-образия,  $\Omega$ -- открытое множество из $(\mathbb{M}^n,g)$.
Тогда любой конечно билипшицевый гомеоморфизм $f:\Omega\to
\mathbb{M}^{n}_*$ является нижним $Q$-гомео\-морфизмом относительно $p$-модуля, $p\in(0,\infty)$, с $Q(x)=K_p(P,f)$.}\end{theorem}

\begin{proof}   Более того, покажем, что
$$
M_p(f\Gamma)\ \geq\ \inf\limits_{\varrho\in ext_p\,
adm\Gamma}\int\limits_{\Omega}\frac{\varrho^p(P)}{K_p(P,f)}\ dV
$$ для любого семейства $\Gamma$ $(n-1)$-мерных
поверхностей $S$ в $\Omega$.

Пусть далее   $B$ -- (борелевское) множество всех точек
$P$ из $\Omega$, где, согласно замечаниям \ref{rem1} и \ref{rem2},  $f$
имеет  дифференциал $f'(P)$ и $J(P,f)\neq 0.$ Известно, что
$B$ является объединением  счетного набора борелевских множеств
$B_l$, $l=1,2,\ldots,$ таких, что $f_l=f|_{B_l}$ билипшицево, см.,
напр., пункт 3.2.2 в \cite{Fe}. Не ограничивая общности, можно считать, что  $B_l$
попарно не пересекаются. Отметим, что $B_0=\Omega\setminus B$ и
$f(B_0)$ имеет нулевую меру в $\mathbb{M}^n$ и
$\mathbb{M}_*^n$, соответственно, ввиду билипшицевости отображения $f$, см.
следствие 8.1 в \cite{MRSY} и замечание \ref{rem1}. Таким образом, по теореме 2.4 в \cite{KPR}, см. теорему 9.1 в \cite{MRSY}, $A_S(B_0)=0$ для $p$-п.в. $S\in\Gamma$ и,
т.к. $f$  -- конечно билипшицевый гомеоморфизм, $A_{S_*}(f(B_0))=0$
для $p$-п.в. $S\in\Gamma,$ где $S_*=f\circ
S$.

Пусть $\varrho_*\in adm\,f\Gamma$, $\varrho_*\equiv0$ вне
$f(\Omega)$, и пусть $\varrho\equiv0$ вне $\Omega$ и
$$\varrho(P)\ =\ \varrho_*(f(P))\,L(P,f)$$ для п.в. $P\in \Omega.$

Рассуждая на каждом $B_l$, по 3.2.20 и 1.7.6 в \cite{Fe} получаем, что
$$\int\limits_{S}\varrho^{n-1}\, d\mathcal{A}\ \geq\
\int\limits_{S_*}\varrho^{n-1}_*\,d\mathcal{A}_*\ \geq\ 1$$ для
$p$-п.в. $S\in\Gamma$ и, таким образом, $\varrho\in
ext_p\,adm\,\Gamma$.

С помощью замены переменных для класса конечно билипшицевых функций,
см., напр., пункт 3.2.5 в \cite{Fe}, и теоремы Лебега получаем, что
$$\sum\limits_{l}\int\limits_{B_l}\frac{\varrho^p(P)}{K_p(P,f)}\ dV\ =
\int\limits_{\Omega}\frac{\varrho^p(P)}{K_p(P,f)}\ dV\ =\
\int\limits_{f(\Omega)}\varrho^p_*(T)\ dV_*,$$ что и приводит к
нужному неравенству. \end{proof}


\section {О граничном поведении конечно билипшицевых го\-мео\-морфизмов} Далее, учитывая теоремы о граничном поведении нижних
$Q$-го\-ме\-о\-мор\-физмов из п. 6 статьи \cite{ARS}, в качестве следствий получаем
ряд теорем о граничном поведении конечно билипшицевых гомеоморфизмов
на гладких римановых многообразиях.

Аналогично \cite{MRSY} говорим, что граница  области $D$ -- {\it слабо плоская в
точке} $P_0\in
\partial D$, если для любого числа $P>0$ и любой окрестности $U$ точки
$P_0$ найдется ее окрестность $V\subset U,$ такая, что
$$M(\Delta(E,F; D))\geq P $$
для любых континуумов $E$ и $F$ в $D$, пересекающих $\partial U$ и
$\partial V.$

Также говорим, что граница области  $D$  {\it сильно
достижима в точке $P_0\in \partial D$}, если для любой окрестности
$U$ точки $P_0$, найдется компакт $E\subset D$, окрестность
$V\subset U$ точки $P_0$ и число $\delta
>0,$ такие, что
$$ M(\Delta(E,F; D))\geq \delta$$
для любого континуума  $F$ в $D$, пересекающего $\partial U$ и
$\partial V.$

Наконец говорим, что граница области $D$  называется {\it сильно достижи\-мой} и {\it слабо
плоской}, если со\-от\-вет\-ствующие свойства имеют место в каждой
точке границы.

Напомним также, что топологическое пространство {\it связно}, если его
нельзя разбить на два непустых открытых множества.  Область $D$
называ\-ется {\it локально связной в точке} $P_0\in\partial D,$ если
для любой окрестности $U$ точки $P_0$ найдется окрестность
$V\subseteq U$ точки $P_0$, такая, что $V\cap D$ связно, ср. \cite[c.
232]{Ku}.

По теореме 6.1 в \cite{ARS} из теоремы 2 получаем следующее заключение.
\begin{theorem} \label{3} {\it Пусть  $D$ локально связна на границе,
$\overline{D}$  компактно, $\partial D_*$ -- слабо плоская. Если
$f:D\to D_*$ -- конечно билипшицевый гомеоморфизм  с $K(P,f)\in
L^{n-1}(D)$, то $f^{-1}$ имеет непрерывное продолжение на
$\overline{D_*}$.}\end{theorem}

\begin{remark}\label{rem3}   Отметим, что здесь условие $K(P,f)\in
L^{n-1}(D)$ нельзя заменить на условие $K(P,f)\in L^{p}(D)$ ни при
каком $p<n-1$, см.,  примеры липшицевых отображений в
доказательстве теоремы 5 в \cite{KovOn}. Однако, здесь достаточно
предполагать, что $K(P,f)\in L^{n-1}(D\cap U)$ для некоторой
окрестности $U$ границы $D$. \end{remark}

По теореме 9.2 в \cite{KR1} из теоремы 2 также имеем следующий результат.
\begin{theorem} \label{4}{\it Пусть  $D$ локально связна на границе,
$\overline{D}$  компактно, $\partial D_*$ -- слабо плоская и
\begin{equation} \label{11.15.5}
\int\limits_{0}^{\delta({x_0})} \frac{dr}
{\|K\|_{n-1}(P_0,r)}=\infty\ \ \ \ \ \ \ \forall\ P_0  \in
\partial D
\end{equation}
для некоторого   $\delta(P_0)\in (0,d(P_0)),$ где $d(P_0):=\sup
\limits_{P \in D}d(P,P_0),$ такого, что
 $B(P_0,\delta(P_0))$ --  нормальная окрестность точки $P_0$ и
$$ \|K\|_{n-1}(P_0,r)=
 \left(\int\limits_{S(P_0,r)}K^{n-1}(P,f)\
d{\cal A}\right)^{\frac{1}{n-1}}.$$Тогда, для любого
конечно билипшицевого гомеоморфизма  $f:D\to D_*$, его обратное
отображение $f^{-1}$ допускает непрерывное продолжение на
$\overline{D_*}$.}\end{theorem}

При этом мы также воспользовались нормальными окрестностями, пред\-ложением 1 и замечанием 1.

Аналогично по лемме 6.1 в \cite{KR1} и теореме 2 имеем:
\begin{lemma} \label{lem3} {\it Пусть\  $D$ \ локально \ связна\  в $P_0\in\partial D$,
$\partial D_*$ сильно достижима хотя бы в одной точке предельного
множества \ \ $C(P_0,f)=\{T\in
\mathbb{M}^{n}_*: T=\lim\limits_{k\to\infty} f(P_k), \ \ P_k\to
P_{0},\ P_k\in D\}$ и $\overline{D_*}$ компактно,
$K(P,f):\mathbb{M}^n \to(0,\infty)$ -- измеримая функция и пусть
$f:D\to D_*$ -- конечно билипшицевый гомеоморфизм  в точке $P_0$.
Если условие (\ref{11.15.5}) выполнено в точке $P_0$, то
$f$ продолжим в $P_0$ по непрерывности.}\end{lemma}

\begin{corollary} \label{cor2} {\it Пусть $D$  локально связна в точке $P_0 \in
\partial D$,
$\partial D_*$ сильно достижима, $\overline{D_*}$ компактно
и пусть  $f:D\rightarrow D_*$ -- конечно билипшицевый гомео\-мор\-физм с
$$K(P,f) = O\left (\log\frac{1}{r}\right)\ \ \ \text{при}\ \  r:=d(P,P_0)\rightarrow0.$$
 Тогда $f$
допускает продолжение в точку $P_0$ по непрерывности на
$(\mathbb{M}^{n}_*,g^{*})$.}
\end{corollary}

Наконец, на основе теоремы \ref{4} и леммы \ref{lem3}, приходим
к следующему заключению.

\begin{theorem} \label{5} {\it Пусть $D$  локально связна на границе,
$\overline{D}$ и $\overline{D_*}$ компактны, $\partial D_*$ -- слабо
плоская.
Тогда любой конечно билипшицевый гомеоморфизм  $f:D\rightarrow D_*$ с условием (\ref{11.15.5})
допускает гомеоморфное продолжение}
 $\overline{f}:\overline{D}\rightarrow \overline{D_*}$.\end{theorem}

\begin{corollary} \label{cor3} {\it Пусть $D$  локально связна на границе,
$\partial D_*$ -- слабо плоская,
 $\overline{D}$ и $\overline{D_*}$ компактны  и пусть $f:D\rightarrow D_*$ -- конечно билипшицевый гомео\-морфизм с
$$ K(P,f) = O\left (\log\frac{1}{r}\right)\ \ \ \ \ \ \ \ \ \text{при}\ \  r:=d(P,P_0)\rightarrow 0\ \
\ \forall \ P_0 \in \partial D.$$
 Тогда   $f$
допускает гомеоморфное продолжение}
$\overline{f}:\overline{D}\rightarrow
\overline{D_*}$.\end{corollary}


\medskip

\textbf{Афанасьева Елена Сергеевна}

Институт прикладной математики и механики НАН Украины

ул. Розы  Люксембург 74, Донецк, 83114.

Рабочий телефон: 311-01-45

\textbf{E-mail:} es.afanasjeva@yandex.ru, smolovayaes@yandex.ru

\end{document}